\documentclass{amsart}
\usepackage{amssymb,latexsym}
\usepackage{amsmath}
\usepackage[dvipdfm]{graphicx}
\usepackage{enumerate}
\newenvironment{enumeratei}{\begin{enumerate}[\upshape (i)]}%
                            {\end{enumerate}}

\theoremstyle{plain}
 \newtheorem{theorem}{Theorem}[section]
 \newtheorem{lemma}[theorem]{Lemma}
 \newtheorem{proposition}[theorem]{Proposition}
 \newtheorem{corollary}[theorem]{Corollary}
\theoremstyle{definition}

\newcommand \tbf [1]{\textbf{#1}}

\newcommand \tolprop {TImC}
\newcommand \weauthors [1] {#1}
\newcommand \weandrefauthors [1] {}

 \usepackage{color}
 
\def\private#1{}
\def\Malcev{Maltsev}

\renewcommand \epsilon{\varepsilon}
\renewcommand\phi{\varphi}
\newcommand \alg[1] {{\mathbf{#1}}} 
\newcommand \var[1] {{\mathcal{#1}}} 
\newcommand \tra {{\pmb\tau}} 
\newcommand \trb {{\pmb\varrho}} 
\newcommand \cra {{\pmb{\vartheta}}} 
\newcommand \set[1] {\{#1\}}
\newcommand \mcond [1] {M(#1)} 
\newcommand \clonalg [1] {\mathfrak A(#1)}
\newcommand \Free [2] {\alg F_{#1}(#2)}
\newcommand \orda {\kappa} 
\newcommand \ordb {\mu} 
\newcommand \con [2] {\textup{con}(#1,#2)}
\newcommand \lsep {\,\mathord{;}\,}
\newcommand \lcol {\mathord{:}\,}
\newcommand \kernel[1]  {\textup{Ker}(#1)}
\newcommand \kalap[1] {{#1}^{\ast}}
\newcommand \vesszo[1] {{#1}^{\bullet}}
\newcommand \gyemant[1] {{#1}^{\diamond}}
\newcommand \realiz [2] {{#1}_{#2}}
\newcommand \kr {\kern 0.24cm }
\newcommand \yes {{$\mathord +$}}
\newcommand \no {{$\mathord -$}}
\begin{document}

\title[Varieties whose tolerances are images of congruences]
{Varieties whose tolerances are homomorphic images of their congruences}
\author[G.\ Cz\'edli]{G\'abor Cz\'edli}
\email{czedli@math.u-szeged.hu}
\urladdr{http://www.math.u-szeged.hu/$\sim$czedli/}
\address{University of Szeged\\Bolyai Institute\\Szeged,
Aradi v\'ertan\'uk tere 1\\HUNGARY 6720}

\author[E.\,W.\ Kiss]{Emil W.~Kiss}
\email{ewkiss@cs.elte.hu}
\urladdr{http://www.cs.elte.hu/$\sim$ewkiss/}
\address{E\"otv\"os University, Department of Algebra and Number Theory, P\'azm\'any P\'eter s\'et\'any 1/c, Budapest, Hungary 1117}

\dedicatory{Dedicated to B\'ela Cs\'ak\'any on his eightieth birthday}

\thanks{This research was supported by the NFSR of Hungary (OTKA), grant numbers   K77432 and K83219.
 }

\private{08A30: Subalgebras, congruence relations,
08B05: Equational logic, \Malcev{} conditions,
08A60 Unary algebras, 
06B10 Lattices/Ideals, congruence relations 
20M99 Semigroups/None of the above, but in this section }


\date{\private{17:55} April 10, 2012}

\begin{abstract}
  The homomorphic image of a congruence is always a
  tolerance (relation) but, within a given variety, a
  tolerance is not necessarily obtained this way.  By a
  \Malcev-like condition, we characterize varieties whose
  tolerances are homomorphic images of their
  congruences~(\tolprop). As corollaries, we prove
  that the variety of semilattices, all varieties of
  lattices, and all varieties of unary algebras have
 \tolprop. We show that a congruence $n$-permutable
    variety has \tolprop\ if and only if it is congruence
    permutable, and construct an idempotent variety with a
    majority term that fails~\tolprop.
\end{abstract}

\maketitle

\section{Introduction}
Let $\alg A=(A;F)$ be a (general) algebra. By a \emph{tolerance} (\emph{relation}) of~$\alg A$ we mean a reflexive, symmetric, and compatible relation $\tra\subseteq A^2$. Transitive tolerances are \emph{congruences}. Tolerances, implicitly or explicitly, often played an important role in the theory of \Malcev{} (also spelled as Mal'cev)  conditions,  for example in  B.~J\'onsson~\cite{JonssonScand} and \weauthors{G.~Cz\'edli, E.\,K.~Horv\'ath and P.~Lipparini~}\cite{CzHL}. Tolerances are particularly useful in lattice theory; partly because the algebraic functions on a finite lattice are just the monotone functions preserving tolerances, see M.~Kindermann~\cite{kindermann}, and also because tolerances play a crucial role in decompositions of modular lattices into maximal complemented intervals, see Ch.~Herrmann~\cite{herrmann} and A.~Day and 
Ch.~Herrmann~\cite{dayherrmann}.

There are two important ways to deal with tolerances. Following I.~Chajda~\cite{chajdablocks}, I.~Chajda, J.~Niederle and B.~Zelinka~\cite{chajdaniederlezelinka}, and \weauthors{G.~Cz\'edli and L.~Klukovits~}\cite{czgklukovits}, one can describe them by their blocks. However, the present paper is devoted to a promising recent approach to tolerances: they can often be characterized as homomorphic images of congruences.

Clearly, see also E.~Fried and G.~Gr\"atzer~\cite{friedgratzer}, if $\phi\colon\alg B\to\alg A$ is a surjective homomorphism and $\tra$ is a tolerance of $\alg B$, then $\phi(\tra)=\set{(\phi(a),\phi(b)): (a,b)\in\tra}$ is a tolerance of~$\alg A$. In  particular, if  $\cra$ is a congruence of~$\alg B$, then $\phi(\cra)$ is a tolerance (but not necessarily a congruence) of $\alg A$. We are interested in varieties of algebras \emph{whose \tbf Tolerances are  homomorphic \tbf{Im}ages of their \tbf Congruences}, \tolprop{} in short. The property \tolprop{} holds in a variety~$\var V$ if for every $\alg A\in \var V$ and each tolerance $\tra$ of~$\alg A$, there exist an algebra $\alg B\in\var V$, a congruence $\cra$ of $\alg B$, and a homomorphism $\phi\colon \alg B\to\alg A$ such that $\tra=\phi(\cra)$. Notice that $\phi$ is necessarily surjective since $\tra$ is reflexive.

Using an old construction discovered  \weauthors{by the first author
}in~\cite{czglperrho}, \weauthors{G.~Cz\'edli and
  G.~Gr\"atzer~}\cite{czggg} proved that the variety of all
lattices satisfies \tolprop{}. Some other varieties
satisfying \tolprop{} have recently been found in
\weandrefauthors{I.~Chajda, G.~Cz\'edli and R.~Hala\v
  s~}\cite{chczgrh} and~\cite{chczgrhref}. In
particular, we know from \cite{chczgrhref} that all
  varieties defined by so-called balanced identities
satisfy \tolprop{}, and each algebra belongs to some variety
satisfying~\tolprop{}.

Our goal is to give a \Malcev-like characterization of the
property~\tolprop{}. This characterization enables us to
find several new results stating that certain varieties,
including all lattice varieties, all unary varieties, and
the variety of semilattices, satisfy~\tolprop{}. In
  the last section of the paper we initiate the
  investigation of the relationship between known \Malcev{}
  conditions and~\tolprop{}.

\section{Characterizing  \tolprop{}}
Let $n\in\mathbb N=\set{1,2,\ldots}$. We say that a variety~$\var V$ satisfies the \emph{\Malcev-like condition} $\mcond n$ if for any pair $(f,g)$ of $2n$-ary terms  such that the identity
\begin{equation*}f(x_0,x_0,x_1,x_1,\ldots,x_{n-1},x_{n-1}) \approx  g(x_0,x_0,x_1,x_1,\ldots,x_{n-1},x_{n-1})
\end{equation*}
holds in $\var V$, there exists a $4n$-ary term $h$ such that the identities 
\begin{align*}
\begin{aligned}
f(x_0,{}&y_0,x_1,y_1,\ldots,x_{n-1},y_{n-1})\cr 
& \approx h(x_0,y_0,x_1,y_1,\ldots,x_{n-1},y_{n-1},x_0,y_0,x_1,y_1,\ldots,x_{n-1},y_{n-1}),
\end{aligned}\\ 
\begin{aligned}
g(x_0,{}&y_0,x_1,y_1,\ldots,x_{n-1},y_{n-1})\cr 
& \approx h(y_0,x_0,y_1,x_1,\ldots,y_{n-1},x_{n-1},x_0,y_0,x_1,y_1,\ldots,x_{n-1},y_{n-1})\phantom,
\end{aligned}
\end{align*}
also hold in $\var V$. Our main goal is to prove the following two theorems.

\begin{theorem}\label{thmmain} For an arbitrary variety~$\var V$ of algebras, the following two conditions are equivalent.
\begin{enumeratei}
\item\label{thmmaina} $\var V$ satisfies \tolprop{}, that is, the tolerances of\/~ $\var V$ are homomorphic images of its congruences.
\item\label{thmmainb} For all $n\in\mathbb N$,  condition $\mcond n$ holds in~$\var V$.
\end{enumeratei}
\end{theorem}

For a variety~$\var V$, let $\clonalg{\var V}$ be the \emph{clone algebra} of~$\var V$ introduced in W.~Taylor~\cite[Definition 2.9]{taylor}. It is a heterogeneous algebra consisting of the equivalence classes of all finitary  terms of $\var V$, where two terms, $t_1$ and~$t_2$, are equivalent if{f} the identity $t_1\approx t_2$ holds in~$\var V$. This heterogeneous algebra is equipped with (heterogeneous)  substitution operations and with constant operations assigning projections. In this terminology, $\mcond n$ becomes a first-order formula in the language of~$\clonalg{\var V}$, and $\mcond n$ holds in~$\var V$ if{f} this first-order formula, also denoted by $\mcond n$, holds in~$\clonalg{\var V}$. Therefore, Theorem~\ref{thmmain} characterizes \tolprop{} by the countably infinite set \hbox{$\set{\mcond n:n\in\mathbb N}$} of first-order formulas in the language of clone algebras of varieties. This raises the question whether we really need infinitely many formulas. The answer and some additional information are given in the following statement.

\begin{theorem}\label{propchargood}\ 
\begin{enumeratei}
\item\label{propchargooda} For $n\in\mathbb N$, $\,\mcond {n+1}$ implies $\mcond n$.
\item\label{propchargoodb} For $n\in\mathbb N$, $\,\mcond {n}$ does not imply $\mcond{n+1}$.
\item\label{propchargoodc} Assume that $\Sigma$ is a finite set of first-order formulas in the language of clone algebras of varieties. Then $\Sigma$, that is the conjunction of all members of~$\Sigma$, \emph{is not equivalent} to~\tolprop{}.
\end{enumeratei}
\end{theorem}

\section{Proving our theorems}
For $n\in\mathbb N$, the list $z_0,\ldots,z_{n-1}$ of the elements (or lists) $z_i$ will often be denoted by $z_i\lcol i<n$, and similar notation applies when $n$ is an ordinal. Lists are concatenated by semicolons. This convention allows us to write terms in a concise form. For example, $\mcond n$ in this notation is the following condition: if
\begin{equation}\label{Mn0}
f(x_i,x_i\lcol  i<n)\approx g(x_i,x_i \lcol  i<n),
\end{equation}
then there exists a $4n$-ary term $h$ such that
\begin{align}\label{Mn1}
f(x_i,y_i\lcol  i<n) &\approx h(x_i,y_i\lcol  i<n\lsep  x_i,y_i\lcol  i<n), \\ 
g(x_i,y_i\lcol  i<n) &\approx h(y_i,x_i\lcol  i<n\lsep  x_i,y_i\lcol  i<n)\text.\label{Mn2}
\end{align}

The algebra freely generated by $X$ in a variety $\var V$
will be denoted by $\Free {\var V}X$.  We consider only
well-ordered free generating sets.  Therefore, if $\orda$
and $\ordb$ denote cardinal numbers and  we write, say,
$X=\set{x_i,y_i\lcol i<\orda\lsep  z_j\lcol j<\ordb}$,
which is the same as $X=\set{x_i:i<\orda}\cup \set{y_i:i<\orda}\cup \set{z_j:j<\ordb}$, then we always assume that 
$\set{x_i:i<\orda}$, 
$\set{y_i:i<\orda}$, and 
$\set{z_j:j<\ordb}$ are pairwise disjoint and each of the equations $x_i=x_j$, $y_i=y_j$, and $z_i=z_j$ 
implies that $i=j$. However, if $a_i$ and~$a_j$ are not necessarily free generators of a free algebra, then $a_i=a_j$ does not imply that $i=j$. Sometimes we allow ``formally infinitary'' terms like $f(x_i,y_i\lcol i<\orda\lsep  z_j\lcol j<\ordb)$; they, of course, depend only on finitely many of their variables. 
The smallest congruence collapsing $a$ and $b$ is denoted by $\con ab$. 
The idea of the following statement goes back to \Malcev~\cite{malcev} and J\'onsson~\cite{JonssonScand}; for the reader's convenience and also to demonstrate how our notation works, we give a short proof.

\begin{lemma}\label{freecongRdescrp} Let\/ $X=\set{x_i,y_i\lcol i<\orda\lsep  z_j\lcol j<\ordb}$ be a nonempty set, let\/ $\var V$ be a variety, and denote by $\cra$ the congruence $\bigvee\set{\con{x_i}{y_i}:i<\orda}$ of $\Free{\var V}X$. Let $f$ and~$g$ be terms over~$X$. 
Then $\cra$ collapses the elements $f(x_i,y_i\lcol
i<\orda\lsep  z_j\lcol j<\ordb)$ and $g(x_i,y_i\lcol
i<\orda\lsep  z_j\lcol j<\ordb)$ of $\Free{\var V}X$ if and only if\/~$\var V$ satisfies the identity 
$f(x_i,x_i\lcol i<\orda\lsep  z_j \lcol j<\ordb)\approx g(x_i,x_i \lcol i<\orda\lsep  z_j\lcol j<\ordb)$.
\end{lemma}

\begin{proof} If the identity holds, then 
\begin{align*}
f(x_i,{}&y_i\lcol i<\orda\lsep  z_j \lcol j<\ordb)  \mathrel{\cra}    f(x_i,x_i\lcol i<\orda\lsep  z_j \lcol j<\ordb) \cr
 & =  g(x_i,x_i\lcol i<\orda\lsep  z_j \lcol j<\ordb) \mathrel{\cra} g(x_i,y_i\lcol i<\orda\lsep  z_j \lcol j<\ordb) \text,
\end{align*}
and the transitivity of~$\cra$ applies. Conversely, if $\cra$ collapses the two elements of $\Free{\var V}X$ in question, then let $Y=\set{x_i\lcol i<\orda\lsep  z_j\lcol j<\ordb}$, and consider the unique homomorphism $\phi\colon \Free{\var V}X\to \Free{\var V}Y$ such that $\phi(x_i)=\phi(y_i)=x_i$ for $i<\orda$ and $\phi(z_j)=z_j$ for $j<\ordb$. Then $\cra\subseteq \kernel{\phi}$ since $\phi$ collapses the pairs that generate~$\cra$. Hence
\begin{align*}
f(x_i,{}&x_i\lcol i<\orda\lsep  z_j \lcol j<\ordb) =\phi\bigl( f(x_i,y_i\lcol i<\orda\lsep  z_j \lcol j<\ordb) \bigr) \cr
&= \phi\bigl( g(x_i,y_i\lcol i<\orda\lsep  z_j \lcol j<\ordb) \bigr) =
g(x_i,x_i\lcol i<\orda\lsep  z_j \lcol j<\ordb),
\end{align*}
which implies that the required identity holds in~$\var V$.
\end{proof}

The following auxiliary statement follows from I.~Chajda~\cite[Lemma 1.7]{chajdabook}; it also follows easily from the observation that the tolerances of~$\alg A$ are just the symmetric subalgebras of~$\alg A^2$ containing the diagonal $\set{(a,a):a\in A}$.

\begin{lemma}\label{lemmatolgen} Assume that\/ $\tra$ is the smallest tolerance of $\alg A=(A;F)$ that contains the pairs $(a_i,b_i)$ for $i<\orda$. Let $(d,e)\in A^2$, and assume that $C=\set{c_j: j<\ordb}\subseteq A$ generates  $\alg A$. Then $(d,e)\in \tra$ if{f} there is  a term~$h$ such that 
\begin{align*}
d=h(a_i,b_i\lcol i<\orda \lsep  c_j\lcol j<\ordb)\,\text{  and }\, 
e=h(b_i,a_i\lcol i<\orda \lsep  c_j\lcol j<\ordb)\text. \qed
\end{align*}
\end{lemma}

\begin{proof}[Proof of Theorem~\ref{thmmain}]
In order to prove that part \eqref{thmmaina} implies part
\eqref{thmmainb}, assume that $n\in \mathbb N$, that   $\var
V$ satisfies
 \tolprop{}, and that $f$ and $g$ are $2n$-ary terms such that the identity $f(x_i,x_i\lcol i<n) \approx g(x_i,x_i\lcol i<n)$ holds in $\var V$.  
Let $X=\set{x_i,y_i\lcol i<n}$ and $\alg F=\Free{\var V}X$. Define $\tra$ as the tolerance generated by $\set{(x_i,y_i): i<n}$. By assumption, there exists a $\alg B\in\var V$, a congruence $\cra$ of $\alg B$, and a surjective homomorphism $\alg B\to\alg F$ such that $\tra=\phi(\cra)$. We can pick elements $a_i,b_i\in B$ such that, for $i<n$, $(a_i,b_i)\in \cra$,\; $\phi(a_i)=x_i$, and $\phi(b_i)=y_i$. Since
\[f(a_i,b_i\lcol i<n) \mathrel{\cra} f(a_i,a_i\lcol i<n) = g(a_i,a_i\lcol i<n)
 \mathrel{\cra}g(a_i,b_i\lcol i<n),
\] 
by applying $\phi$ we conclude that 
$\bigl( f(x_i,y_i\lcol i<n) ,g(x_i,y_i\lcol i<n)   \bigr)\in\tra$.
Therefore, applying Lemma~\ref{lemmatolgen} with $C=X$ and using that an equation of two terms on the free generators is an identity that holds in~$\var V$, we obtain a $2n$-ary term~$h$ such that Identities~\eqref{Mn1} and~\eqref{Mn2} hold in~$\var V$. Hence part~\eqref{thmmaina} implies part~\eqref{thmmainb}.

To prove the converse implication, assume that  $\mcond n$ holds in $\var V$ for all $n\in\mathbb N$, $\alg A=(A;F)\in \var V$, and $\tra=\set{(a_i,b_i): i<\orda}$ is a tolerance of~$\alg A$. Here $\orda$ is an ordinal.  By reflexivity, $A=\set{a_i:i<\orda}$, but usually this is a redundant enumeration of~$A$. We are going to find a congruence preimage~$\cra$ of~$\tra$ in two steps: first we construct 
a ``free'' tolerance preimage~$\trb$ of~$\tra$, and then a congruence preimage~$\cra$ of~$\trb$. 

Let $Y=\set{u_i,v_i\lcol i<\orda}$, 
$\alg G=\Free{\var V}Y$, and let $\trb$ be the tolerance generated by $\set{(u_i,v_i):i<\orda}$. Consider the surjective homomorphism $\phi\colon\alg G\to \alg A$ such that $\phi(u_i)=a_i$ and $\phi(v_i)=b_i$ for $i<\orda$. Since $\phi(\trb)$ is a tolerance of~$\alg A$ and contains the pairs $(a_i,b_i)=\bigl(\phi(u_i),\phi(v_i)\bigr)$, we have that $\tra\subseteq \phi(\trb)$. 
To show the converse inclusion, take a pair in~$\phi(\trb)$. It is of the form 
$
\bigl(t_1(a_i,b_i\lcol i<\orda), t_2(a_i,b_i\lcol i<\orda)  \bigr)
$ where $t_1$ and $t_2$ are terms and $\bigl(t_1(u_i,v_i\lcol i<\orda), t_2(u_i,v_i\lcol i<\orda)  \bigr)\in\trb$.
Applying  Lemma~\ref{lemmatolgen} with $C=Y$, we obtain a term $t_3$ such that 
\begin{align*}
t_1(u_i,v_i\lcol i<\orda) &= t_3(u_i,v_i\lcol i<\orda \lsep  u_i,v_i\lcol i<\orda  ),\cr
t_2(u_i,v_i\lcol i<\orda) &= t_3(v_i,u_i\lcol i<\orda \lsep  u_i,v_i\lcol i<\orda  )\text.
\end{align*} 
Since $\phi$ transfers these two equations to
\begin{align*}
t_1(a_i,b_i\lcol i<\orda) &= t_3(a_i,b_i\lcol i<\orda \lsep  a_i,b_i\lcol i<\orda  ),\cr
t_2(a_i,b_i\lcol i<\orda) &= t_3(b_i,a_i\lcol i<\orda \lsep  a_i,b_i\lcol i<\orda  ),
\end{align*} 
it follows (directly or from Lemma~\ref{lemmatolgen}) that 
$\bigl( t_1(a_i,b_i\lcol i<\orda), t_2(a_i,b_i\lcol
i<\orda)\bigr)\in \tra$. Therefore,
$\tra=\phi(\trb)$. In this first step we did not use   condition~$\mcond n$.

Next, let $\set{(c_j,d_j):j<\ordb}=\trb$, $X=\set{x_j,y_j\lcol j<\ordb}$, and $\alg F=\Free{\var V}X$. Consider the congruence $\cra=\bigvee\set{\con{x_j}{y_j}:j<\ordb}$ of~$\alg F$, and let $\psi\colon \alg F\to \alg G$ be the surjective homomorphism defined by $\psi(x_j)=c_j$ and $\psi(y_j)=d_j$ for $j<\ordb$. Clearly, $\trb\subseteq\psi(\cra)$. To prove the converse inclusion, 
take a pair in~$\cra$. It is of the form
 $\bigl( p(x_j,y_j\lcol  j<\ordb),  q(x_j,y_j\lcol  j<\ordb)\bigr)\in\cra$, where $p$ and~$q$ are terms. We have to prove that the pair 
\[\bigl( p(c_j,d_j\lcol  j<\ordb),  q(c_j,d_j\lcol  j<\ordb)\bigr) =\bigl( \psi(p(x_j,y_j\lcol  j<\ordb)), \psi( q(x_j,y_j\lcol  j<\ordb))\bigr)\] 
belongs to $\trb$. 
We obtain from Lemma~\ref{freecongRdescrp} that
\begin{equation}\label{fKlPsw}
\text{the identity $
p(x_j,x_j\lcol  j<\ordb)\approx  q(x_j,x_j\lcol  j<\ordb)$ holds in~$\var V$.}
\end{equation}
Since the terms $p$ and~$q$ depend on finitely many variables and the original ordering of variables is irrelevant, we can assume 
that  $p(x_j,y_j\lcol  j<\ordb)\approx f(x_i,y_i\lcol i<n)$ and $q(x_j,y_j\lcol  j<\ordb)\approx g(x_i,y_i\lcol i<n)$ hold in $\var V$ for some $n\in\mathbb N$ and some $2n$-ary terms $f$ and~$g$. Hence \eqref{fKlPsw} turns into
\begin{equation*}
\text{$f(x_i,x_i\lcol i<n)\approx g(x_i,x_i\lcol i<n)$ holds in~$\var V$, }
\end{equation*}
and all we have to prove is that
\begin{equation}\label{HRmZkyzs}
\bigl(f(c_i,d_i\lcol i<n), g(c_i,d_i\lcol i<n)\bigr) \in \trb\text.
\end{equation}
Let $h$ be a $4n$-ary term provided by~$\mcond n$. Since 
\begin{align*}
f(c_i,d_i\lcol i<n) &= h(c_i,d_i\lcol i<n \lsep c_i,d_i\lcol i<n), \cr 
g(c_i,d_i\lcol i<n) &= h(d_i,c_i\lcol i<n \lsep c_i,d_i\lcol i<n),
\end{align*}
and $h$ preserves~$\trb$, \eqref{HRmZkyzs} follows. This shows that $\trb=\psi(\cra)$.

Finally, the composite map $\phi\circ\psi\colon \alg F\to \alg A$, $z\mapsto \phi\bigl(\psi(z)\bigr)$ is a surjective homomorphism, and 
$(\phi\circ\psi)(\cra)=\phi\bigl(\psi(\cra)\bigr)= \phi(\trb)=\tra$. That is, 
$\var V$ satisfies~\tolprop{}.
\end{proof}

\begin{proof}[Proof of Theorem~\ref{propchargood}] To prove part \eqref{propchargooda}, assume that $\mcond {n+1}$ holds in a variety~$\var V$, and $f$ and~$g$ are $2n$-ary terms such that the identity $f(x_i,x_i\lcol i<n)\approx 
g(x_i,x_i\lcol i<n)$ holds in~$\var V$. By adding two fictitious 
 variables, we define two $2(n+1)$-ary terms as follows: $\kalap f(x_i,y_i\lcol i\leq n)=f(x_i,y_i\lcol i<n)$ and $\kalap g(x_i,y_i\lcol i\leq n)=g(x_i,y_i\lcol i<n)$. Since $\kalap f(x_i,x_i\lcol i\leq n)\approx \kalap g(x_i,x_i\lcol i\leq n)$ clearly holds in $\var V$, $\mcond {n+1}$ gives a $4(n+1)$-ary term $\kalap h$ such that the identities 
\begin{align}\label{kalapfgHwz}
\begin{aligned}
f(x_i,y_i\lcol i < n)\approx \kalap f(x_i,y_i\lcol i\leq n)&\approx \kalap h(x_i,y_i\lcol i\leq n  \lsep x_i,y_i\lcol i\leq n ),\cr 
g(x_i,y_i\lcol i < n)\approx\kalap g(x_i,y_i\lcol i\leq n) &\approx \kalap h(y_i,x_i\lcol i\leq n  \lsep x_i,y_i\lcol i\leq n )
\end{aligned}
\end{align}
hold in $\var V$. Define a $4n$-ary term~$h$ by letting
$h(u_i,v_i\lcol i<n\lsep x_i,y_i\lcol i<n)$ to be $\kalap
h(u_i,v_i\lcol i<n\lsep u_{n-1},v_{n-1}\lsep  x_i,y_i\lcol
i<n \lsep x_{n-1},y_{n-1})$.
It follows from \eqref{kalapfgHwz} that, with this~$h$,  \eqref{Mn1} and~\eqref{Mn2} hold in~$\var V$. Hence $\mcond {n}$ holds in~$\var V$, proving part~\eqref{propchargooda} of Theorem~\ref{propchargood}.

To prove part~\eqref{propchargoodb}, we construct a variety $\var V$ generated by an algebra $\alg A=(A;f,g)$ such that $\mcond {n}$ holds but $\mcond {n+1}$ fails in~$\var V$. Let $A=\set{0,1,\ldots, 2n+2}$, and denote $A\setminus\set 0$ by $A^+$. We define a $2(n+1)$-ary operation $f$  on~$A$ by the following rule:
\begin{align*}
f(a_i,b_i\lcol i\leq n)=
\begin{cases} 1,&\text{if }\set{a_i,b_i\lcol i\leq n}=A^+,\cr
0,&\text{otherwise.}
\end{cases}
\end{align*}
Similarly, we also define a $2(n+1)$-ary operation $g$ on~$A$ as follows:
\begin{align*}
g(a_i,b_i\lcol i\leq n)=
\begin{cases} 2,&\text{if }\set{a_i,b_i\lcol i\leq n}=A^+,\cr
0,&\text{otherwise.}
\end{cases}
\end{align*}
This way we have defined $\alg A=(A;f,g)$ and~$\var V$. 

The identity  $f(x_i,x_i\lcol i\leq n)\approx g(x_i,x_i\lcol i\leq n)$ clearly holds in~$\var V$ since both sides induce the constant $A^{n+1}\to\set 0$ map in~$\alg A$.     Suppose for a contradiction that $\mcond {n+1}$ holds in~$\var V$. Then there exists a $4(n+1)$-ary term $h$ such that  \eqref{Mn1} and~\eqref{Mn2} hold in $\var V$ with the above-defined $f$ and~$g$. 
Then $h$ is not a projection since neither~$f$, nor~$g$ is projection. Therefore the term $h$ has an outermost operation, which is either $f$ or~$g$. 
If the outermost operation is~$f$, then the term function~$\realiz h{\alg A}$, induced by $h$ on $\alg A$, cannot take the value 2, whence \eqref{Mn2} fails in $\alg A$. Similarly, if the outermost operation is~$g$, then \eqref{Mn1} fails by the 1-2 symmetry. Therefore, $\mcond {n+1}$ fails in~$\var V$. 

To show that $\mcond {n}$ holds in $\var V$, observe that any two $2n$-ary terms that are not projections are equivalent in $\var V$ since they induce the same constant map $A^{2n}\to\set 0$ in $\alg A$. 
Therefore, for any two $2n$-ary terms $\gyemant f$ and $\gyemant g$, either none of them is a projection and 
 we can let $\gyemant h(u_i,v_i\lcol i<n\lsep x_i,y_i\lcol i<n)= \gyemant f(x_i,y_i\lcol i<n)$, or both are projections and we can trivially find an appropriate $\gyemant h$. This proves that $\mcond {n}$ holds in $\var V$.

Next, to prove part~\eqref{propchargoodc}, suppose for a contradiction that \tolprop{} is equivalent to a finite~$\Sigma$. We can assume that $\Sigma=\set{\sigma}$ is a singleton since otherwise we can form the conjunction of all members of~$\Sigma$. Taking Theorem~\ref{thmmain} into account, we obtain that $\sigma$ is equivalent to $\set{\mcond k:k\in \mathbb N}$. 
Notice that, by introducing unary relations instead of components and replacing heterogeneous operations by usual relations, heterogeneous 
 algebras can easily be described 
by usual relational systems. Thus the compactness theorem is valid for heterogeneous algebras, and we conclude that  there is a finite set $S\subseteq \mathbb N$ such that $\set{\mcond k: k\in S}$ implies~$\sigma$.
Let $n$ be the largest element of~$S$. By part~\eqref{propchargooda},  $\mcond n$ in itself implies~$\sigma$ and, therefore,~\tolprop{}. Hence, again by Theorem~\ref{thmmain}, $\mcond n$ implies $\mcond {n+1}$, which contradicts part~\eqref{propchargoodb}.  
\end{proof}

\section{Applications}

Next, we give some consequences of
Theorem~\ref{thmmain}. In \weauthors{G.~Cz\'edli and
    G.~Gr\"atzer~}\cite{czggg} it is proved that the variety
  of all lattices satisfies
  \tolprop{}. Theorem~\ref{thmmain} yields the much stronger
  statement that every variety of lattices satisfies this
  property.

\begin{corollary}\label{corLatvR} Assume that $\var V$ is a variety with the following properties.
\begin{enumeratei}
\item\label{corLatvRa} $\var V$ has two binary terms, $\vee$ and $\wedge$, that satisfy the lattice axioms;
\item\label{corLatvRb} for each operation symbol $f$, say $n$-ary, $\var V$ satisfies the identity 
\[f(x_i\wedge y_i\lcol i<n) \wedge f(x_i\lcol i<n)\approx f(x_i\wedge y_i\lcol i<n)\text.
\]
$($In other words, all operations are monotone with respect to the lattice reduct.$)$
\end{enumeratei}
Then $\var V$ satisfies~\tolprop{}.
\end{corollary}
In virtue of this corollary, every variety of lattices
satisfies \tolprop{}. So does every variety of lattices with
involution; see, for example, \weauthors{I.~Chajda  and   G.~Cz\'edli~}\cite{chczInv} for the definition.
\begin{proof}[Proof of Corollary~\ref{corLatvR}]
Assume that $f$ and~$g$ are $2n$-ary terms such that Identity~\eqref{Mn0} holds in~$\var V$. Define a $4n$-ary term $h$ as follows:
\begin{align*}
h(u_i,v_i\lcol i<n\lsep x_i,y_i\lcol i<n) =
f(x_i\wedge u_i, y_i\wedge v_i\lcol i<n)\vee
g(x_i\wedge v_i, y_i\wedge u_i\lcol i<n)\text.
\end{align*}
Using \eqref{Mn0} and the assumption that the terms of~$\var V$ are monotone with respect to the lattice order~$\leq$ induced by $\wedge$ and~$\vee$, we conclude that 
\begin{align*}
h(x_i,y_i\lcol i<n\lsep x_i,y_i\lcol i<n) 
&\approx f(x_i\wedge x_i, y_i\wedge y_i\lcol i<n)\vee
g(x_i\wedge y_i, y_i\wedge x_i\lcol  i<n)\cr
&\approx f(x_i, y_i\lcol i<n)\vee
f(x_i\wedge y_i, x_i\wedge y_i\lcol  i<n)\cr
&\approx f(x_i, y_i\lcol i<n)
\end{align*} 
holds in~$\var V$. Similarly, 
\begin{align*}
h(y_i,x_i\lcol i<n\lsep x_i,y_i\lcol i<n) 
&\approx f(x_i\wedge y_i, y_i\wedge x_i\lcol i<n)\vee
g(x_i\wedge x_i, y_i\wedge y_i\lcol i<n)\cr
&\approx g(x_i\wedge y_i, x_i\wedge y_i\lcol i<n)\vee
g(x_i,y_i\lcol i<n)\cr
&\approx g(x_i, y_i\lcol i<n)\text.
\end{align*} 
Therefore, $\mcond n$ holds in~$\var V$, and Theorem~\ref{thmmain} applies.
\end{proof}
The next three corollaries exemplify how to apply Theorem~\ref{thmmain} for varieties in which the terms and identities are easy to handle. 
While the proof above allowed us to enrich the lattice structure with further monotone operations, the next proof 
seems not to allow a similar enrichment.

\begin{corollary}\label{corSmLaTs}
The variety of semilattices satisfies~\tolprop{}.
\end{corollary}

\begin{proof} Up to equivalence,  each semilattice term is characterized by the variables occurring in it. Assume that $f$ and~$g$ are $2n$-ary terms such that \eqref{Mn0} holds in the variety~$\var V$ of semilattices. We define an appropriate $4n$-ary semilattice term $h(u_i,v_i\lcol i<n\lsep x_i,y_i\lcol i<n)$ by specifying which variables occur in it. 
This is done for each $i$ separately by
Table~\ref{slattableaux}; notice that, by \eqref{Mn0}, $f$
contains at least one of $x_i$ and~$y_i$  if{f} so does~$g$.
Thus we obtain an $h$ witnessing that $\mcond n$ holds in~$\var V$, and Theorem~\ref{thmmain} applies. 
\end{proof}

\begin{table}[h]
\begin{tabular}{ | c | c | c | c | c |c | c | c |} \hline
\multicolumn{2}{| c }{Occurs in $f$} & \multicolumn{2}{| c } {Occurs in $g$} & \multicolumn{4}{| c |} { Occurs in $h$  }  \\ \hline
\kr$x_i$\kr & \kr$y_i$\kr & \kr$x_i$\kr & \kr$y_i$\kr & \kr$u_i$\kr & \kr$v_i$\kr& \kr$x_i$\kr &   \kr$y_i$\kr   \\ \hline
\no&\no&    \no&\no&       \no&\no&      \no&\no  \\ \hline
\yes&\no&    \yes&\no&       \no&\no&      \yes&\no  \\ \hline
\no&\yes&    \no&\yes&       \no&\no&      \no&\yes  \\ \hline
\yes&\no&    \no&\yes&       \yes&\no&      \no&\no  \\ \hline
\no&\yes&    \yes&\no&       \no&\yes&      \no&\no  \\ \hline
\yes&\yes&    \yes&\yes&       \no&\no&      \yes&\yes  \\ \hline
\yes&\yes&    \yes&\no&       \no&\yes&      \yes&\no  \\ \hline
\yes&\yes&    \no&\yes&       \yes&\no&      \no&\yes  \\ \hline
\yes&\no&    \yes&\yes&       \yes&\no&      \yes&\no  \\ \hline
\no&\yes&    \yes&\yes&       \no&\yes&      \no&\yes  \\ \hline
\multicolumn{8}{ c } {   }  \\ 
\end{tabular}
\caption{Defining $h$ in the proof of Corollary~\ref{corSmLaTs}}
\label{slattableaux}
\end{table}

The next statement is a particular case of the result in \cite{chczgrhref} on balanced varieties. The proof we give here is entirely different from that in \cite{chczgrhref}.

\begin{corollary}\label{simtypVAr} For each tolerance~$\tra$ of an algebra~$\alg A$, there exist an algebra $\alg B$, a~congruence $\cra$ of\/~$\alg B$, and a homomorphism $\phi\colon \alg B\to \alg A$ such that $\phi(\cra)=\tra$. 
\end{corollary}

\begin{proof} Let $\var V$ be the class of all algebras similar to (have the same type as)~$\alg A$. The corollary asserts that $\var V$ satisfies \tolprop{}. Let $f$ and~$g$ be $2n$-ary terms such that  \eqref{Mn0} holds in~$\var V$. 
Let $\kalap f=\kalap f(z_j\lcol j<s)$ be the term we obtain
from $f$ by distinguishing its variables. For example, if
$f(x_0,y_0)= \bigl((y_0x_0)(x_0y_0)\bigr)x_0$ (in the language of one
binary operation),  then $\kalap f(z_i\lcol i<5)=\bigl((z_0z_1)(z_2z_3)\bigr)z_4$. 
Define $\kalap g$ analogously.
Since only trivial identities hold in~$\var V$, the terms $f(x_i,x_i\lcol i<n)$ and 
$g(x_i,x_i\lcol i<n)$ are the same (equal sequences of symbols) and, moreover, $\kalap f=\kalap g$. Let $\kalap h=\kalap f$.

\begin{table}[h]
\begin{tabular}{ | c | c | c |} \hline
{$z_j$ in $f$} &  {$z_j$ in $g$} &  {$z_j$ in $h$}  \\ \hline 
$x_i$ & $x_i$ & $x_i$ \\ \hline 
$y_i$ & $y_i$ & $y_i$ \\ \hline 
$x_i$ & $y_i$ & $u_i$ \\ \hline 
$y_i$ & $x_i$ & $v_i$ \\ \hline 
\multicolumn{3}{ c } {   }  \\ 
\end{tabular}
\caption{Defining $h$ in the  proof of Corollary~\ref{simtypVAr}}
\label{tableaux2}
\end{table}

By substituting one of the elements of $\set{u_i,v_i,x_i,y_i\lcol i<n}$ for~$z_j$ in $\kalap h$ according to Table~\ref{tableaux2}, we clearly obtain a term $h(u_i,v_i\lcol i<n\lsep x_i,y_i\lcol i<n)$  witnessing that 
 $\mcond n$ holds in~$\var V$.
\end{proof}

A variety is \emph{unary} if all of its basic operations are at most unary.

\begin{corollary}\label{CorUnry} Every unary variety satisfies~\tolprop{}. 
\end{corollary}

\begin{proof} We modify the proof of Corollary~\ref{simtypVAr} as follows. Assume that \eqref{Mn0} holds in~$\var V$. Since every term of~$\var V$ depends on at most one variable, there exist a~$j$ and a unary term~$\vesszo f$
such that $f(x_i,y_i\lcol i<n)\approx \vesszo f(x_j)$ or $f(x_i,y_i\lcol i<n)\approx \vesszo f(y_j)$ holds in~$\var V$. Similarly, there exist a~$k$ and a unary term~$\vesszo g$
such that $g(x_i,y_i\lcol i<n)\approx \vesszo g(x_k)$ or $g(x_i,y_i\lcol i<n)\approx \vesszo g(y_k)$ holds in~$\var V$. 

Assume first that $j\neq k$. Then \eqref{Mn0} yields that
$\vesszo f(x_j)\approx \vesszo g(x_k)$ holds in $\var V$,
and so $f(x_i,y_i\lcol i<n)$ and $g(x_i,y_i\lcol i<n)$
induce the same constant function on each $\alg A\in\var V$.
Hence we can define~$h$ by 
$h(u_i,v_i\lcol i<n\lsep x_i,y_i\lcol i<n)=f(x_i,y_i\lcol i<n)$. 

Secondly, assume that~$j=k$. Then \eqref{Mn0} yields that $\vesszo f(x_j)\approx \vesszo g(x_j)$ holds in~$\var V$. Clearly, we can  define~$h$ according to Table~\ref{tableaux3}.
\begin{table}[h]
\begin{tabular}{ | c | c | c |} \hline
{$f$ depends on} &  {$g$ depends on} &  {$h(u_i,v_i\lcol i<n\lsep x_i,y_i\lcol i<n)$}  \\ \hline 
$x_j$ & $x_j$ & $\vesszo f(x_j)$ \\ \hline 
$x_j$ & $y_j$ & $\vesszo f(u_j)$ \\ \hline 
$y_j$ & $x_j$ & $\vesszo f(v_j)$ \\ \hline 
$y_j$ & $y_j$ & $\vesszo f(y_j)$ \\ \hline 
\multicolumn{3}{ c } {   }  \\ 
\end{tabular}
\caption{Defining $h$ in the  proof of Corollary~\ref{CorUnry}}
\label{tableaux3}
\end{table}
\end{proof}

A systematic survey of known varieties with
\tolprop{} is not pursued in this paper.  We note
  that, as opposed to the previous corollaries,
  Theorem~\ref{thmmain} is not always the most convenient
  tool to prove the \tolprop{} property. For example, every
variety defined by a set of balanced identities satisfies
\tolprop{} by \weandrefauthors{I.~Chajda, G.~Cz\'edli,
  R.~Hala\v s, and yyy}\cite{chczgrhref}. In particular, so
do the variety of all semigroups and that of all commutative
semigroups.  We wonder what the situation is with
  other important varieties of semigroups.

\section{Natural \Malcev{} conditions and \tolprop{}}

We have shown in Corollary~\ref{corLatvR} that every
  variety of lattices with additional monotone operations 
  satisfies~\tolprop{}. A natural generalization would be to
  consider varieties with a \emph{majority term} (a
  ternary term $m$ such that the identities $m(x,x,y)\approx
  x$, $m(x,y,x)\approx x$, and $m(y,x,x)\approx x$ hold in~$\var
  V$). 
Lattices have such a term. Also, they constitute and \emph{idempotent} variety. (A variety $\var V$ is idempotent if $\,t(x,\ldots,x)\approx x$ is an identity  of~$\var V$ for every basic operation~$t$). The following   example shows that these conditions together are still not sufficient to establish~\tolprop{}.

\begin{proposition}\label{propcExMpl3}
There exists an idempotent variety $\var V$ with a majority
term such that \tolprop{} fails in~$\var V$ and $\var V$ is generated by a three element algebra.
\end{proposition}

\begin{proof} Let $A=\set{0,1,2}$, and define an algebra $\alg A=(A;f,g,m)$, where $f$ and~$g$ are idempotent  quaternary operations and $m$ is a ternary majority operation defined as follows. 
\begin{align*}
f(x_i\lcol i<4)&=
\begin{cases} 1,&\text{if }(x_i\lcol i<4)\in\set{(1,1,1,1),(1,0,0,2)}   ,\cr
2,&\text{if }(x_i\lcol i<4)=(2,2,2,2),\cr
0,&\text{otherwise;}
\end{cases}\cr 
g(x_i\lcol i<4)&=
\begin{cases} 2,&\text{if }(x_i\lcol i<4)\in\set{(2,2,2,2),(1,0,0,2)}   ,\cr
1,&\text{if }(x_i\lcol i<4)=(1,1,1,1),\cr
0,&\text{otherwise;}
\end{cases}\cr 
m(x_0,x_1,x_2)&=
\begin{cases} 0,&\text{if }|\set{x_0,x_1,x_2}|=3 ,\cr
j,&\text{if }|\set{i: x_i=j}|\geq 2\text{.}
\end{cases} 
\end{align*}
Then~$\var V$, the variety generated by~$\alg A$, is an
idempotent variety with a majority term. 
Consider the relation $\tra
=A^2\setminus\set{(1,2),(2,1)}$. We show that $\tra$ is a tolerance of~$\alg A$.

Suppose for a contradiction that $f$ does not preserve~$\tra$. Then there are $(a_i,b_i)\in \tra$  such that $\bigl(f(a_i\lcol i<4), f(b_i\lcol i<4)\bigr)\notin \tra$. By symmetry, we can assume that $\bigl(f(a_i\lcol i<4), f(b_i\lcol i<4)\bigr)=(1,2)$. However, then $(a_0,b_0)=(1,2)\notin\tra$ is a contradiction. Hence $f$ preserves~$\tra$. 
So does  $g(x_i\lcol i<4)$ since it is the ``1-2 dual'' of $f(x_{3-i}\lcol i<4)$.
Next, suppose for a contradiction that $(a_i,b_i)\in \tra$ for $i<3$ but, say, 
$\bigl(m(a_i\lcol i<3), m(b_i\lcol i<3)\bigr)=(1,2)$. Then
at least two of the $a_i$ equal~1 and at least two of the
$b_i$ equal~2. Thus there is an $i\in\set{0,1,2}$ such that
$(a_i,b_i)=(1,2)\notin\tra$, which is a
contradiction. 
Therefore, $\tra$ is indeed a tolerance of~$\alg A$. 

Finally we show that \tolprop{} fails in~$\var V$. Suppose for a contradiction that $\alg B\in \var V$, $\cra$ is a congruence of $\alg B$, $\phi\colon \alg B\to \alg A$ is a homomorphism, and $\phi(\cra)=\tra$. Pick $(a,b), (c,d)\in\cra$ such that $\bigl((\phi(a),\phi(b))=(1,0)\in\tra$ and $\bigl((\phi(c),\phi(d)\bigr)=(0,2)\in\tra$. Observe that the identity $f(x,x,y,y)\approx g(x,x,y,y)$ holds in~$\var V$ since it holds in~$\alg A$. Therefore 
\[
f(a,b,c,d)\mathrel{\cra} f(a,a,d,d) = g(a,a,d,d)\mathrel{\cra} g(a,b,c,d),
\]
and we obtain the following contradiction:
\begin{align*}
(1,2)&=\bigl(f(1,0,0,2),\,g(1,0,0,2) \bigr)\cr
&=\bigl(f(\phi(a),\phi(b),\phi(c),\phi(d)  ),\,g(\phi(a),\phi(b),\phi(c),\phi(d)) \bigr)\cr
&=\bigl(\phi(f(a,b,c,d)),\,\phi(g(a,b,c,d))\bigr)  \in\phi(\cra)=\tra\text. \qedhere
\end{align*}
\end{proof}

Another frequently considered \Malcev{} condition is congruence
  permutability. Each congruence permutable variety~$\var
V$ satisfies~\tolprop{} since every tolerance of an algebra
in~$\var V$ is known to be a congruence, see
J.\,D.\,H.~Smith~\cite{smith} (explicitly) or
H.~Werner~\cite{werner} (implicitly). As an illustration, we
give a new proof, based on
Theorem~\ref{thmmain}.

\begin{corollary} Every congruence permutable variety satisfies \tolprop{}.
\end{corollary}  

\begin{proof} Assume  
that $\var V$ is a congruence permutable variety. By a classical result of A.\,I.~Mal'\-cev~\cite{malcev}, $\var V$ has a \Malcev{} term~$p$, that is, a ternary term~$p$ such that  $p(x,x,y)\approx y\approx p(y,x,x)$ holds in~$\var V$. Assume that $f$ and~$g$ satisfies \eqref{Mn0} in~$\var V$. Let 
\begin{align*}
h(u_i,v_i\lcol i<n\lsep x_i,y_i\lcol i<n)
=p\bigl( f( x_i,y_i\lcol i<n),  f( x_i,u_i\lcol i<n),  g( x_i,u_i\lcol i<n)\bigr)
\text.\end{align*}
Obviously, this $h$ witnesses that $\mcond n$ holds in~$\var V$, and Theorem~\ref{thmmain} applies.
\end{proof}

The strength of the property~\tolprop{} is very well shown
by the following theorem, which refutes a possible
generalization.

\begin{theorem}\label{nperm} A congruence $n$-permutable
  variety has~\tolprop{} if and only if it is congruence permutable.
\end{theorem}

\begin{proof} 
  The previous corollary shows one direction, so suppose
  that a variety $\var V$ is $n$-permutable and satisfies
  \tolprop{}. By the results of Hagemann and
  Mitschke~\cite{HaMi}, there exist ternary terms
  $p_1,\ldots, p_n$ such that the following are identities
  of~$\var V$:
\begin{align*}
&x \approx p_1(x,y,y)\,,\\
&p_i(x,x,y) \approx p_{i+1}(x,y,y) \text{ for $1\le i\le n-1$,}\\
&p_n(x,x,y) \approx y\,.
\end{align*}
We shall construct a term~$p$ such that $\var V$ satisfies
the identities $x \approx p(x,y,y)$ and $p(x,x,y) \approx
p_3(x,y,y)$. This replaces $p_1$ and $p_2$ above, implying
that the variety $\var V$ is actually $n-1$-permutable. Then
we shall be done by induction on~$n$.

Define
\[
f(x,u,v,y)=p_1(x,u,y)\text{ and } g(x,u,v,y)=p_2(x,v,y)\,.
\]
Then $f(x,x,y,y)\approx g(x,x,y,y)$ is an identity of~$\var
V$, since this reduces to the identity $p_1(x,x,y)\approx
p_2(x,y,y)$. Thus $\mcond 2$ implies the existence of  an $8$-ary term~$h$ satisfying the following identities:
\begin{align*}
&h(x,u,v,y,x,u,v,y) \approx f(x,u,v,y)\approx p_1(x,u,y)\,,\\
&h(u,x,y,v,x,u,v,y) \approx g(x,u,v,y)\approx p_2(x,v,y)\,.
\end{align*}
Since $p_1$ does not depend on $v$, we can substitute $v\to
y$ in the first identity, and similarly, $u\to x$ in the
second identity, so we get that
\begin{align}\label{w1}
h(x,u,y,y,x,u,y,y) &\approx p_1(x,u,y)\,,\\
h(x,x,y,v,x,x,v,y) &\approx p_2(x,v,y)\label{w2}
\end{align}
still hold in~$\var V$. Finally, let
\[
p(a,b,c)=h(a,b,c,b,a,b,b,c)\,.
\]
Then the substitution $u\to y$ in~(\ref{w1}) gives
\[
p(x,y,y)=h(x,y,y,y,x,y,y,y)\approx p_1(x,y,y)\approx x\,,
\]
and the substitution $v\to x$ in~(\ref{w2}) yields
\[
p(x,x,y)=h(x,x,y,x,x,x,x,y)\approx p_2(x,x,y)\approx p_3(x,y,y)\,,
\]
proving the theorem.
\end{proof}

Theorem~\ref{nperm} leads to further examples of varieties without \tolprop{}. For example, the variety of implication algebras is 3-permutable, see A.~Mitschke~\cite{mitschke}, while that of $n$-Boolean algebras, see E.\,T.~Scmidt~\cite{scmidtbook} and see also \cite{HaMi}, is $(n+1)$-permutable. Hence it follows from  Theorem~\ref{nperm} that these (non-idempotent) varieties do not satisfy \tolprop{} since they are not congruence permutable. 

In view of Theorem~\ref{nperm}, 
it would be interesting to see if 
there is a connection between \tolprop{} and other famous \Malcev{} conditions.


\end{document}